\documentclass[a4paper]{article}
\usepackage[utf8]{inputenc}
\usepackage[T1]{fontenc}

\usepackage{amsmath,amsfonts}
\let\amsast=\ast
\usepackage{amssymb}
\usepackage{mathtools}
\usepackage{wasysym} 
\usepackage{amsthm}
\usepackage[auth-sc]{authblk} 

\usepackage[svgnames]{xcolor}

\usepackage[]{geometry}

\usepackage[]{hyperref}
\hypersetup{colorlinks=true,
  citecolor=DarkBlue,
  linkcolor=DarkBlue,
  linktoc=page,
  urlcolor=black,
}

\title{Sharp trace Gagliardo-Nirenberg-Sobolev inequalities for convex cones, and convex domains}
\author{Simon Zugmeyer}
\affil{Univ Lyon, Université Claude Bernard Lyon 1, CNRS UMR 5208, Institut Camille Jordan, 43 blvd. du 11 novembre 1918, F-69622 Villeurbanne cedex, France. \texttt{\href{mailto:zugmeyer@math.univ-lyon1.fr}{zugmeyer@math.univ-lyon1.fr}}}
\date{}

\makeatletter 
\g@addto@macro\bfseries{\boldmath}
\makeatother

\newtheorem{thrm}{Theorem}[section]
\newtheorem{prop}[thrm]{Proposition}
\newtheorem{coro}[thrm]{Corollary}
\newtheorem{lemm}[thrm]{Lemma}

\theoremstyle{definition}
\newtheorem{defi}[thrm]{Definition}
\newtheorem*{claim}{Claim}

\theoremstyle{remark}
\newtheorem*{rmrk}{Remark}

\newcommand{\R}{\mathbb{R}}

\newcommand{\N}{\mathbb{N}}

\newcommand{\p}{{p^\amsast}}
\newcommand{\eps}{\varepsilon}
\newcommand{\mc}{\mathcal}

\newcommand{\ic}{\mathbin{\Square}}
\newcommand{\dom}{\operatorname{dom}}
\newcommand{\epi}{\operatorname{epi}}
\newcommand{\lip}{\operatorname{Lip}}

\DeclarePairedDelimiter{\norm}{\lVert}{\rVert} 
\DeclarePairedDelimiter{\abs}{\lvert}{\rvert}

\makeatletter
\renewcommand\d[1]{\mspace{2mu}\mathrm{d}#1\@ifnextchar\d{\mspace{-1mu}}{}}
\let\oldabs\abs
\def\abs{\@ifstar{\oldabs}{\oldabs*}}
\let\oldnorm\norm
\def\norm{\@ifstar{\oldnorm}{\oldnorm*}}
\makeatother

\begin{document}

\maketitle

\begin{abstract}
  We find a new sharp trace Gagliardo-Nirenberg-Sobolev inequality on convex cones, aswell as a sharp weighted trace Sobolev inequality on epigraphs of convex functions. This is done by using a generalized Borell-Brascamp-Lieb inequality, coming from the Brunn-Minkowski theory. 
\end{abstract}


\section{Introduction and main results}

The classical Sobolev inequality states that, for any function \(f\) sufficiently smooth and decaying fast enough at infinity, defined on the Euclidean space \(\R^n\) with \(n\geq 2\) (for instance, \(f\in\mc C^\infty_c(\R^n)\)), and for any \(p\in[1,n)\),
\begin{equation}\label{eq:classical_sobolev}
\norm{f}_{L^\p(\R^n)} \leq C \norm{\nabla f}_{L^p(\R^n)},\quad \p=\frac{pn}{n-p},
\end{equation}
Furthermore, equality is reached in inequality \eqref{eq:classical_sobolev} if \(f\) can be written
\[
f(x)=\left(1+\norm x^{p/(p-1)}\right)^{\frac{p-n}{p}},
\]
up to a translation, a rescaling, and multiplication by a constant, where \(\norm .\) is the Euclidean norm. This was proved by Talenti \cite{talenti} and Aubin \cite{aubin} independently for $p=2$. The Sobolev inequality can be seen as a corollary of a more general inequality, the Gagliardo-Nirenberg inequality, which states that
\begin{equation}\label{eq:classical_GN}
  \norm{f}_{L^q(\R^n)}\leq C\norm{\nabla f}_{L^p(\R^n)}^\theta \norm{f}_{L^r(\R^n)}^{1-\theta},
\end{equation}
for any \(p\in[1,n)\), \(q,r\in[1,+\infty]\), \(\theta\in[0,1]\)  such that
\[
\frac{1}{q}=\left(\frac{1}{p}-\frac{1}{n}\right)\theta +\frac{1-\theta}{r};
\]
whence the case \(\theta=1\) is exactly the Sobolev inequality. This family of inequalities has been notably investigated for \(p=2\) by del Pino and Dolbeault \cite{delpino_dolbeault}, who have not only found an explicit sharp constant, but also proved that there is equality if, and only if, \(f\) has the form
\[
f(x)=\left(1+\norm x^2\right)^{\frac{2}{2-q}},
\]
up to, once again, a translation, a rescaling, and multiplication by a constant.

As Bobkov and Ledoux \cite{bobkov_ledoux} showed, these sharp inequalities can be reached within the framework of the Brunn-Minkovski theory \cite{schneider}. With this approach, the sharp inequality follows in the more general case where the Euclidean norm is replaced by a generic norm on \(\R^n\), which is a result already proved by Cordero-Erausquin, Nazaret, and Villani using optimal transport \cite{CENV}. This makes sense, since the Brunn-Minkovski inequality directly implies the isoperimetric inequality, which is famously equivalent to the sharp Sobolev inequality with \(p=1\) (for a nice overview on this subject, see Osserman's article on the isoperimetric inequality \cite{osserman}).

The key tool Bobkov and Ledoux use is an extended Borell-Brascamp-Lieb inequality, a quick proof of which using optimal transport is given by Bolley, Cordero-Erausquin, Fujita, Gentil and Guillin \cite{BCEFGG}. 
For a bit of context, let us state the Brunn-Minkoski inequality: for any compact nonempty subsets $A$ and $B$ in $\R^n$, and any $t\in\left[0,1\right]$
\[
\abs{tA+(1-t)B}^{1/n}\geq t\abs A^{1/n} + (1-t)\abs B^{1/n},
\]
where $\abs.$ denotes the Lebesgue measure on $\R^n$. This is to say that the volume, to the power $1/n$, is concave with respect to the Minkowski sum, defined by $A+B=\{a+b,\,(a,b)\in A\times B\}$.
The classical Borell-Brascamp-Lieb inequality \cite{borell}\cite{brascamp_lieb}, just like the isoperimetric inequality, follows from the Brunn-Minkowski inequality. It is, in some sense, its functional counterpart: let \(t\in [0,1]\) and \(u,v,w:\R^n\to(0;+\infty]\) such that for all \(x,\,y\in\R^n\),
\[
w((1-t)x+ty) \leq \left((1-t)(u(x))^{-1/n}+t(v(y))^{-1/n}\right)^{-n},
\]
then
\[
\int w\geq \min\left(\int u,\int v\right).
\]
Playing with the exponents and normalizing this inequality gives the following reformulation of the Borell-Brascamp-Lieb inequality:
let \(g\), \(W\), and \(H:\R^n\to(0,+\infty]\), and \(t\in[0,1]\), such that \(\int g^{-n}=\int W^{-n}=1\) and
\[
\forall x,y\in\R^n,\quad H((1-t)x+ty)\leq (1-t)g(x)+tW(y)
\]
then
\begin{equation}\label{eq:classical_BBL}
  \int H^{-n} \geq 1.
\end{equation}
Applying this inequality to the greatest function \(H\) meeting these criteria allows us to prove that
\begin{equation}\label{eq:derived_classical_BBL}
\int W^\amsast(\nabla g)g^{-n-1}\geq 0,
\end{equation}
where \(W^*\) is the Legendre transform of \(W\). This inequality, as we will see in the next section, turns out to be equivalent to the Borell-Brascamp-Lieb inequality we use here. This might look like it is to be expected, because of the semigroup structure that underlies the theorem, but is actually a little bit surprising, because said semigroup is not quite linear. The equivalence between the more general theorems with which we work here remains an open question.

Inequality \eqref{eq:derived_classical_BBL} can, in turn, be used to prove sharp Sobolev-type inequalities, but in the end proves to be limited as it does not allow to reach the full range of Gagliardo-Nirenberg inequalities showcased by del Pino and Dolbeault \cite{delpino_dolbeault}. Thus, a better inequality to work with is the following extension of the Borell-Brascamp-Lieb inequality
\begin{thrm}\label{th:better_BBL}
  Let \(n\geq 2\), and \(t\in[0,1]\). Let \(g\), \(W\), and \(H:\R^n\to(0,+\infty]\) be measurable functions such that \(\int g^{-n}=\int W^{-n}=1\) and
    \begin{equation}\label{eq:better_BBL}
      \forall \,x,\,y\in\R^n,\quad H((1-t)x+ty)\leq (1-t)g(x)+tW(y)
    \end{equation}
    then
    \[
    \int H^{1-n} \geq (1-t)\int g^{1-n}+t\int W^{1-n}.
    \]  
\end{thrm}
With this theorem, we are able to prove sharp trace-Sobolev inequalities on convex domains. More specifically, we prove sharp trace Sobolev in some convex domains, and sharp trace Gagliardo-Nirenberg inequalities in convex cones. In what follows, \(\norm.\) is a norm on \(\R^n\), and \(\norm._*\) is the dual norm, defined by \(\norm x_*=\sup_{\norm y=1}x\cdot y\). In \(L^q\) norms of vector functions, the dual norm \(\norm._*\) will be used.
Let $\varphi:\R^{n-1}\to\R$ be a convex function such that $\varphi(0)=0$. We consider functions defined on $\varphi$'s epigraph, that is $\Omega=\{(x_1,x_2)\in\R^{n-1}\times\R,\,x_2\geq\varphi(x_1)\}$. We say that $\Omega$ is a convex cone whenever $\varphi$ is positive homogeneous of degree $1$: for all $t>0$ and $x_1\in\R^{n-1}$, $\varphi(tx_1)=t\varphi(x_1)$. 

\begin{thrm}[Sharp trace Gagliardo-Nirenberg inequality]\label{th:trace_GN_cone}
  Let \(a\geq n>p>1\), and \(\Omega=\{(x_1,x_2)\in\R^{n-1}\times\R,\,x_2\geq\varphi(x_1)\}\) be a convex cone. There exists a positive constant $D_{n,p,a}(\Omega)$ such that for any non-negative function \(f\in C^\infty_c(\Omega)\),
  \begin{equation}\label{eq:trace_GN_cone}
    \left(\int_{\R^{n-1}}f^{q}(x,\varphi(x))\d x\right)^{1/q}\leq D_{n,p,a}(\Omega) \norm{\nabla f}_{L^p(\Omega)}^\theta \norm{f}_{L^q(\Omega)}^{1-\theta},
  \end{equation}
  where
  \[
  \theta=\frac{a-p}{p(a-n-1)+n},\quad q=p\frac{a-1}{a-p}.
  \]
  Furthermore, when \(f(x)=\norm{(x_1,x_2+1)}^{-\frac{a-p}{p-1}}\), then \eqref{eq:trace_GN_cone} is an equality.
\end{thrm}
The fact that there exists a function for which the equality is reached means that the constant $D_{n,p,a}(\Omega)$ may be computed explicitly.
Choosing \(a=n\), Theorem \ref{th:trace_GN_cone} immediately yields the sharp trace Sobolev inequality as a corollary:
\begin{coro}[Sharp trace Sobolev inequality]
  Let \(n>p>1\), and \(\Omega=\{(x_1,x_2)\in\R^{n-1}\times\R,x_2\geq\varphi(x_1)\}\) be a convex cone. There exists a positive constant $D_{n,p}(\Omega)=D_{n,p,n}(\Omega)$ such that for any non-negative function \(f\in C^\infty_c(\Omega)\),
  \begin{equation}\label{eq:trace_sob_cone}
    \left(\int_{\R^{n-1}}f^{p\frac{n-1}{n-p}}(x,\varphi(x))\d x\right)^{\frac{n-p}{p(n-1)}}\leq D_{n,p}(\Omega) \norm{\nabla f}_{L^p(\Omega)},
  \end{equation}
  Furthermore, when \(f(x)=\norm{(x_1,x_2+1)}^{-\frac{n-p}{p-1}}\), then \eqref{eq:trace_sob_cone} is an equality.
\end{coro}
The case $\Omega=\R^{n}_+$ has already been studied by Nazaret \cite{nazaret}.

If we only assume \(\Omega\) to be convex, we prove, under some growth criteria on \(\Omega\), the following sharp weighted trace Sobolev inequality:
\begin{thrm}[Sharp trace Sobolev inequality]\label{th:trace_sob}
  Let \(n>p>1\), and \(\Omega=\{(x_1,x_2)\in\R^{n-1}\times\R,x_2\geq\varphi(x_1)\}\) be a convex set. Assume that there exist some constants \(C>0\) and \(R>0\) such that
  \[
  \forall\,x_1\in\R^{n-1} \text{ s.t. }  \norm{x_1}>R,\quad \abs{x_1\cdot\nabla\varphi(x_1)}\leq C\norm{(x_1,\varphi(x_1))}.
  \]
  Then, there exists a positive constant $D'_{n,p}(\Omega)$ such that for any nonnegative function \(f\in C^\infty_c(\Omega)\),
  \begin{equation}\label{eq:trace_sob}
    \int_{\R^{n-1}}f^{p\frac{n-1}{n-p}}(x,\varphi(x))P(x)\d x\leq D'_{n,p}(\Omega) \left(\int_\Omega\norm{\nabla f}_*^p\right)^{\frac{n-1}{n-p}}
  \end{equation}
  where \(P(x)=1+\varphi(x)-x\cdot\nabla\varphi(x)\).
  Furthermore, when \(f(x)=\norm{(x_1,x_2+1)}^{-\frac{n-p}{p-1}}\), then \eqref{eq:trace_sob} is an equality.
\end{thrm}
Once again, $D'_{n,p}(\Omega)$ can be computed explicitly. This inequality may be surprising, since the weight \(P\) can (and usually is, whenever \(\Omega\) is not a cone) negative outside a compact neighbourhood of \(0\), but it is still sharp. For instance, with the set defined by \(\varphi(x)=\norm x^2\), the weight becomes \( P(x)=1-\norm x^2\), which happens to be negative outside the unit ball. One may define $\partial\Omega_+\subset\partial\Omega$ such that $\partial\Omega_+=\{(x_1,\varphi(x_1)), \, P(x_1)> 0\}$. In that case, inequality \eqref{eq:trace_sob} restricted to functions \(f\in C^\infty_c(\mathring \Omega \cup \partial\Omega_+)\) becomes a regular weighted inequality, with a positive weight.

In the next section, we first study the infimal convolution, which is the key tool in the proof of Theorems \ref{th:trace_GN_cone} and \ref{th:trace_sob}. Once these are established, we prove the claimed equivalence between the classical Borell-Brascamp-Lieb inequality \eqref{eq:classical_BBL} and its differentiated formulation \eqref{eq:derived_classical_BBL}, within some limitations. Next, in section \ref{se:sharp_GNS}, we move on to prove the main Theorems \ref{th:trace_GN_cone} and \ref{th:trace_sob}, starting from an improved version of the Borell-Brascamp-Lieb inequality. The technical details, which will be glided over in these sections, can be found in the comprehensive appendix \ref{se:admissibility}, at the end of the paper.


\section{Generalities}

Let \(t\in[0,1)\). To use Theorem \ref{th:better_BBL}, instead of considering any \(H\) such that
\[
\forall \,x,\,y\in\R^n,\quad H((1-t)x+ty)\leq (1-t)g(x)+tW(y),
\]
we may well choose the greatest such function. That is,
\[
H(z)=\inf_{\substack{x,y\in\R^n\\(1-t)x+ty=z}}\{(1-t)g(x)+tW(y)\},
\]
or, writing \(h=t/(1-t)\),
\[
\frac{H(z)}{1-t}=\inf_{y\in\R^n}\left\{g\left(\frac{z}{1-t}-hy\right)+hW(y)\right\}.
\]
This formula, being explicit, allows for some properties to be brought to light. It motivates the definition, and the study, of the so-called infimal convolution:
\begin{defi}
  Let \(f,g:\R^n\to\R\cup\{+\infty\}\). Their infimal convolute \(f\ic g:\R^n\to\R\cup\{+\infty\}\) is defined by
  \[
  (f\ic g)(x)=\inf_{y,z\in\R^n}\{f(y)+g(z), \,y+z=x\}=\inf_{y\in\R^n}\{f(y)+g(x-y)\}.
  \]
  The infimal convolution of \(f\) with \(g\) is said to be \emph{exact at \(x\)} if the infimum is achieved, and \emph{exact} if it is exact everywhere.  
\end{defi}
With this definition, and whenever \(h=t/(1-t)>0\), the greatest function \(H\) in Theorem \ref{th:better_BBL} is given by
\[
H(z)=(1-t)\inf_{y\in\R^n}\left\{g\left(\frac{z}{1-t}-y\right)+hW(y/h) \right\} = (1-t)\left(g\ic hW(./h)\right)(z/(1-t)),
\]
we thus define 
\[
Q_h^W(g)=g\ic hW(./h)=x\mapsto \inf_{y\in\R^n}\{g(x-y)+hW(y/h)\}.
\]
Using \(Q_h^W\) in Theorem \ref{th:better_BBL}, inequality \eqref{eq:better_BBL} becomes
\begin{equation}\label{eq:BBLn}
  \int Q_h^W(g)^{1-n} \geq \int g^{1-n} + h\int W^{1-n}
\end{equation}
but there exists a slightly more general version of this inequality, namely Theorem \ref{th:BBL}, which we will use in section \ref{se:sharp_GNS}.

To begin with, let us first showcase some properties of the infimal convolution.


\subsection{The general infimal convolution}

This subsection is here to build some intuition about infimal convolution, before proving specific results useful for the study of \(Q_h^W\).
\begin{defi}
  With any function \(f:\R^n\to \R\cup\{+\infty\}\), we associate its
  \begin{itemize}
  \item essential domain (usually shortened to domain), \(\dom f=\{x\in\R^n,\,f(x)<+\infty\}\);
  \item epigraph, \(\epi f=\{(x,\alpha)\in \R^n\times\R,\, f(x)\leq \alpha\}\);
  \item strict epigraph, \(\epi_s f=\{(x,\alpha)\in \R^n\times\R,\, f(x)< \alpha\}\).
  \end{itemize}
  Furthermore, the function \(f\) is said to be proper if it is not equal to the constant \(+\infty\).
\end{defi}
With these definitions, we highlight in the next proposition the link between infimal convolution of functions and Minkowski sum of sets, classically defined for two sets $A,B$ by $A+B=\{a+b,\,(a,b)\in A\times B\}$.
\begin{prop}
  Let \(f,g:\R^n\to\R\cup\{+\infty\}\). Then
  \begin{itemize}
  \item \(\dom f\ic g = \dom f+\dom g\);
  \item \(\epi_s f\ic g = \epi_s f + \epi_s g\);
  \item \(\epi f\ic g \supset \epi f + \epi g\), and equality holds if, and only if, the infimal convolution is exact at each \(x\in\dom f\ic g\).
  \end{itemize}
\end{prop}
Proof of this proposition and more in-depth details on infimal convolutions can be found in Thomas Strömberg's thesis \cite{stromberg}. The more delicate question of regularity of the infimal convolution is only addressed in subsection \ref{se:regularity} in the particular study of $Q_h^W(g)$. That is because there is not \emph{one} natural set of assumptions ensuring regularity, so it really depends on the goal, which, here, is that \(Q_h^W(g)\) should be smooth enough to prove Sobolev inequalities. We only prove the following lemma in the most general case, since it is very useful.
\begin{lemm}\label{le:exactness}
  Let \(f,g:\R^n\to\R\cup\{+\infty\}\) be lower semicontinuous functions. If \(f\) is nonnegative and \(g\) is coercive, that is,
  \[
  \lim_{\norm x\to+\infty}g(x)=+\infty,
  \]
  then \(f\ic g\) is exact.
\end{lemm}
\begin{proof}
  Fix \(x\in\R^n\). Consider \(\psi:\R^n\to\R\cup\{+\infty\},\,y\mapsto f(x-y)+g(y)\) and assume that there exists \(y_0\) such that \(\psi(y_0)<+\infty\): \(\psi\) is lower semicontinuous, and greater than \(g\), thus tends to \(+\infty\) as \(\norm y\) goes to \(+\infty\). As such, \(\{y\in\R^n,\psi(y)\leq\psi(y_0)\}\) is closed and bounded, thus compact. Now, let \((y_n)\subset\{\psi\leq\psi(y_0)\}\) be a minimizing sequence, \(\lim_{n\to+\infty}\psi(y_n)=\inf_{y\in\R^n}\{\psi(y)\}\). By compactness, we can assume that the sequence \((y_n)\) converges towards \(z\in\R^n\), and by lower semicontinuity, \(-\infty<\psi(z)\leq \lim_{n\to+\infty}\psi(y_n)=\inf_{y\in\R^n}\{\psi(y)\}\), thus the infimum is finite and is actually a minimum. 
  If such a \(y_0\) does not exist, then \(f\ic g(x)=+\infty\), and the infimum is also reached.
\end{proof}


\subsection{Regularity of the inf-convolution \texorpdfstring{$Q_h^W(g)$}{Q\_h(g)}}\label{se:regularity}

We begin here the specific study of \(Q_h^W(g)=g\ic hW(./h)\). 
The study of the regularity of \(Q_h^W(g)\) with respect to \(h>0\) is crucial, because we would like to differentiate inequality \eqref{eq:BBLn} with respect to \(h\). Let us first state some classical results about the Legendre transform. The proofs can be found in Evans' book, \cite[p.120]{evans}, and Brézis' book, \cite[p.10]{brezis}. 
\begin{defi}
  The Legendre transform of \(W\) is defined by
  \[
  W^*(y)=\sup_{x\in\R^n}\{x\cdot y-W(x)\}\in\overline \R.
  \]
\end{defi}
By definition, \(W^*\) is a lower semicontinuous convex function, but it is not always proper. For \(W^*\) to be well behaved, we have to assume a little bit more about \(W\). In fact, it is enough to assume \(W\) to be lower semicontinuous: indeed, if \(W:\R^n\to \R\cup\{+\infty\}\) is a lower semicontinuous proper convex function, then \(W^*\) is also a lower semicontinuous proper convex function, and \((W^*)^*=W\). The infimal convolution is not only closely related to Minkovski sums, but also to Legendre transforms, as the next lemma shows.
\begin{lemm}\label{le:hamilton_jacobi}
  Let \(g,W:\R^n\to(-\infty,+\infty]\) be two measurable functions.
  If \(g\) is nonnegative and almost everywhere differentiable on its domain \(\dom g=\Omega_0\) (with nonempty interior), and \(W\) grows superlinearly,
  \[
  \lim_{\abs x\to+\infty}\frac{W(x)}{\abs x}=+\infty,
  \]
  then for almost every \(x\in \mathring \Omega_0\), $h\mapsto Q_h^W(g)(x)$ is differentiable at $h=0$, and
  \[
  \left.\frac{\partial}{\partial h}\right\vert_{h=0}Q_h^W(g)(x)=-W^*(\nabla g(x)),
  \]
  where \(W^*\) is the Legendre transform of \(W\).
\end{lemm}
\begin{proof}
  Let \(\Omega_1=\dom W\), and fix \(x\in\mathring\Omega_0\) such that the differential of \(g\) at \(x\) exists. Let \(y\in \Omega_1\). For \(h>0\) sufficiently small, \(x-hy\in \Omega_0\), and we get, by definition of \(Q_h^W(g)\),
  \[
  \frac{Q_h^W(g)(x)-g(x)}{h}\leq \frac{g(x-hy)-g(x)}{h}+W(y).
  \]
  Taking the superior limit when \(h\to 0\) yields
  \[
  \limsup_{h\to 0}\frac{Q_h^W(g)(x)-g(x)}{h}\leq -\nabla g(x)\cdot y+W(y).
  \]
  This being true for any \(y\in\Omega_1\), we may take the infimum to find that
  \[
  \limsup_{h\to 0}\frac{Q_h^W(g)(x)-g(x)}{h} \leq -W^*(\nabla g(x)).
  \]
  Conversely, fix \(e\in \Omega_1\), and \(h_0>0\) such that \(\overline{ B(x,h_0\norm e)}\in\mathring\Omega_0\). For \(h\in(0,h_0)\), define
  \[
  \Omega_{x,h}=\{y\in \Omega_1,\, hW(y)\leq g(x-he)+hW(e)\};
  \]
  note that \(e\in\Omega_{x,h}\). We claim that \(\limsup_{h\to0}\{h\norm y,\,y\in \Omega_{x,h}\}=0\). Indeed, if \(y\in \Omega_{x,h}\), then
  \[
  h\norm y\frac{W(y)}{\norm y}\leq g(x-he)+hW(e)\leq \sup_{z\in\overline{B(x,h_0\norm e)}} g(z)+h_0W(e).
  \]
  Now, when \(h\) goes to \(0\), either \(\limsup\norm y< +\infty\), or \(\limsup\norm y= +\infty\); in both cases, since \(\lim_{\abs y\to+\infty}\frac{W(y)}{\norm y}=+\infty\), the claim is proved. Notice now that for all \(h\in(0,h_0)\), \(Q_h^W(g)(x)\leq g(x-he)+hW(e)\), hence \(Q_h^W(g)(x)=\inf_{y\in \Omega_{x,h}}\{\dots\}\). Thus,
  \begin{align*}
    \frac{Q_h^W(g)(x)-g(x)}{h}&=\inf_{y\in \Omega_{x,h}}\left\{\frac{g(x-hy)-g(x)}{h}+W(y)\right\}\\
    &=\inf_{y\in \Omega_{x,h}}\left\{-\nabla g(x)\cdot y+y\cdot\eps_x(hy)+W(y)\right\}
  \end{align*}
  where \(\eps_x(z)\to 0\) when \(\norm z\to 0\). Let \(1\geq\eta>0\); the claim proves that there exists \(h_\eta\in(0,h_0)\) such that for all \( 0 < h < h_\eta\), \(\forall y\in \Omega_{x,h}\), \(\norm{\eps_x(hy)}\leq \eta\). Thus,
  \begin{align*}
    \frac{Q_h^W(g)(x)-g(x)}{h}&\geq\inf_{y\in \Omega_{x,h}}\left\{-\nabla g(x)\cdot y-\eta\norm y+W(y)\right\}\\
    &=\inf_{\substack{y\in \Omega_{x,h}\\y\in B(0,R)}}\{\dots\}\\
    &\geq \inf_{y\in \Omega_{x,h}}\{-\nabla g(x)\cdot y+W(y)\}-R\eta\\
    &\geq -W^*(\nabla g(x))-R\eta,
  \end{align*}
  where \(R\) was chosen such that \(\norm y\geq R \implies W(y)\geq (\norm{\nabla g(x)}+1)\norm y+W(e)-\nabla g(x)\cdot e\). Finally, taking the inferior limit of this inequality, and noticing that the result stays true for any \(0<\eta\leq 1\), we may conclude (since \(R\) is independent from \(\eta\)) that
  \[
  \lim_{h\to 0}\frac{Q_h^W(g)(x)-g(x)}{h}= -W^*(\nabla g(x)).
  \]
\end{proof}

This differentiation result is enough to prove the main theorems contained in section \ref{se:sharp_GNS}, but we can go a little bit further with more assumptions on \(g\) and \(W\). Assuming \(W\) to be convex bestows upon \(Q_h^W\) a semigroup structure:
\begin{lemm}\label{le:semigroup}
  Assume that \(g:\R^n\to[0,+\infty]\) is lower semicontinuous, and that \(W\) is a lower semicontinuous proper convex function such that \(\lim_{\norm x\to+\infty}W(x)=+\infty\). Then, for all \(x\in\R^n\) and \(0<s<h\),
  \begin{align*}
    Q_h^W(g)(x)&=\min_{y\in\R^n}\{g(x-hy)+hW(y)\}\\
    &= Q_{h-s}^W(Q_s^W(g))(x).
  \end{align*}  
\end{lemm}
\begin{proof}
  Exactness was already proved in Lemma \ref{le:exactness}.
  Notice that
  \begin{align*}
    Q_{h-s}^W(Q_s^W(g))(x)&=\inf_{y\in\R^n}\inf_{z\in\R^n}\{g(x-(h-s)y-sz)+(h-s)W(y)+sW(z)\}\\
    &\leq \inf_{y\in\R^n}\{g(x-hy)+hW(y)\}=Q_h^W(g)(x).
  \end{align*}
  Conversely, let \(y\in\R^n\), and choose \(z\in\R^n\) such that
  \[
  Q_s^W(g)(x-(t-s)y)=g(x-sz)+sW(z).
  \]
  Then, by convexity,
  \begin{align*}
    Q_t^W(g)(x)&\leq g(x-(t-s)y-sz)+tW\left(\frac{t-s}{t}y+\frac{s}{t}z\right)\\
    &\leq g(x-(t-s)y-sz)+(t-s)W(y)+sW(z)\\
    &= (t-s)W(y) + Q_s^W(g)(x-(t-s)y).
  \end{align*}
  Taking the infimum over \(y\in\R^{n}\) proves that \(Q_t^W(g)(x)\leq Q_{h-s}^W(Q_s^W(g))(x)\), and thus there is equality.
\end{proof}

We want to investigate if some kind of regularity is preserved under the operation of infimal convolution. The answer is yes, under certain specific conditions. We will also provide an example showcasing regularity loss, emphasizing the delicate nature of this question. Work on this subject already exists, notably in Evans' book \cite[p.~128]{evans}, where there is a global Lipschitz assumption, or in Villani's book \cite[Theorem 30.30]{villani2009}, where functions are bounded. However, such assumptions are at odds with the goals we aim for  here, as ultimately, we want \(g^{-\alpha}\) to be integrable for some exponant \(\alpha>0\).

Let us study the case where \(g\) and \(W\) are finite \emph{everywhere}. 
\begin{lemm}\label{le:full_lipschitz_continuity}
  Let \(g,W:\R^n\to\R\). If \(g\) is nonnegative, locally Lipschitz continuous, and \(W\) is convex and coercive, then \((h,x)\mapsto Q_h^W(g)\) is locally Lipschitz continuous.
\end{lemm}
\begin{proof}
  In order to prove the full local Lipschitz continuity, we must first localize the \(\operatorname{arginf}\) of the infimal convolution. Fix \(\rho>0\), \(\eta>0\), and let \(x,x'\in B(0,\rho)\) and \(0< h< \eta\).
  Consider the set
  \[
  \Omega_{x,h}\coloneqq\{y\in\R^n, g(x-y)+hW(y/h)\leq g(x)+hW(0)\}.
  \]
  We claim that, by positivity of \(g\), and convexity of \(W\), the set is bounded. Indeed, since \(W\) is convex and coercive, there exists \(R>0\) and \(m>0\) such that 
  \[
  \norm y >R \implies W(y)\geq m\norm y.
  \]
  If \(y\in\Omega_{x,h}\), then either \(\norm y\leq hR \leq \eta R\), or \(\norm y > hR\) and then \(g(x)+hW(0) \geq hW(y/h) \geq m\norm y\). Invoking continuity of \(g\), we may prove the claim, and conclude that there exists \(R_{\rho,\eta}\), independent from \(x\) and \(h\), such that \(\Omega_{x,h}\subset B(0,R_{\rho,\eta})\). 

  Let us now prove the local Lipschitz continuity with respect to \(x\). The functions \(g\) and \(W\) are assumed continuous, and so the infimal convolution is exact, and there exists \(y\in\R^{n}\) such that \(Q_h^W(g)(x)=g(x-y)+hW(y/h)\). Necessarily, \(\norm y \leq R_{\rho,\eta}\), so
  \begin{align*}
    Q_h^W(g)(x')-Q_h^W(g)(x) &= \inf_{y'\in\R^n}\left\{g(x'-y')+hW(y'/h)\right\}-g(x-y)-hW(y/h)\\
    &\leq g(x'-y)-g(x-y) \\
    &\leq \left(\lip_{B(0,\rho+R_{\rho,\eta})}g\right)\norm{x-x'},
  \end{align*}
  where \(\lip_A f\coloneqq\sup_{x\neq x'\in A}\{\abs{f(x)-f(x')}/\norm{x-x'}\}\). By symmetry, we conclude that 
  \[
  \abs{Q_h^W(g)(x')-Q_h^W(g)(x)}\leq \left(\lip_{B(0,\rho+R_{\rho,\eta})}g\right)\norm{x-x'},
  \]
  hence the local Lipschitz continuity with respect to \(x\).
  
  Now,
  \begin{align*}
    Q_h^W(g)(x)-g(x)&=\inf_{y\in B(0,R_{\rho,n})}\{g(x-y)-g(x)+hW(y/h)\} \\
    &\geq  \inf_{y\in B(0,R_{\rho,\eta})}\left\{-(\lip_{B(0,\rho+R_{\rho,\eta})}g)\norm{y}+hW(y/h)\right\}\\
    &=h\inf_{z\in B(0,R_{\rho,\eta}/h)}\left\{-\lambda\norm z +W(z)\right\}\\
    &\geq -h\sup_{z\in \R^n}\left\{\lambda\norm z -W(z)\right\}\\
    &\geq -h \sup_{t\in B(0,\lambda)} W^*(t),
  \end{align*}
  where \(\lambda=\lip_{B(0,\rho+R_{\rho,\eta})}g\). Conversely, by definition,
  \[
  Q_h^W(g)(x)-g(x)\leq hW(0),
  \]
  and thus \(\abs{Q_h^W(g)(x)-g(x)}\leq Ch\), where \(C=\max\{W(0),\sup_{t\in B(0,\lambda)} W^*(t)\}\). Note that $C$ is finite because $W^*$ is, by definition, convex and finite on $\R^n$, thus continuous. Finally, using the semigroup property \(Q_{h+s}^W(g)=Q_h^W(Q_s^W(g))\) and the fact that the Lipschitz constant with respect to \(x\) is uniformly bounded by \(\lip_{B(0,\rho+R_{\rho,\eta})}\) for \(0<h<\eta\), we may conclude for the full local Lipschitz continuity.
\end{proof}

The above lemma is a slight generalization of the following proposition:
\begin{prop}\label{pr:lipschitz_continuity}
  Let \(f,g:\R^n\to\R\) be lower semicontinuous functions. If \(f\) is nonnegative, locally Lipschitz continuous, and \(g\) is coercive, then \(f\ic g\) is locally Lipschitz continuous.
\end{prop}
Here, we do not need any convexity assumption, which was only used to prove Lipschitz continuity with respect to the $(n+1)$th variable, $h$. Also, note here that it is important for \(f\) and \(g\) to be finite \emph{everywhere}, which will not be the case in sections \ref{se:sharp_GNS} and appendix \ref{se:admissibility}. In order for $f\ic g$ to be locally Lipschitz continuous, further assumptions are needed on \(f\) and \(g\), in particular on their domain. For example, if \(\dom f=\{x_0\}\), then \(f\ic g=f(x_0)+g(\,.-x_0)\), so it already seems necessary that both \(f\) and \(g\) be at least locally Lipschitz continuous. However, this is not sufficient. Consider for example the following functions \(f\) and \(g\), defined on \(\R^2\) by
\[
f(x_1,x_2)=\begin{cases} 1 &\text{if }x_1\in[0,1],\,x_2=0,  \\ 
  1-x_2 &\text{if } x_1=0,\,x_2\in[0,1], \\ 
  +\infty &\text{otherwise,} \end{cases}\quad\text{and}\quad 
g(x_1,x_2)=\begin{cases} 0 &\text{if }x_1\in[0,1],\,x_2=0,  \\
  +\infty &\text{otherwise,} \end{cases}
\]
then
\[
(f\ic g)(x_1,x_2)=\begin{cases} 1&\text{if } x_1\in(0,1],\,x_2\in[0,1], \\
  1-x_2 &\text{if } x_1=0,\,x_2\in[0,1], \\
  0 &\text{if } x_1=0,\,x_2\in[1,2], \\
  +\infty &\text{otherwise} \end{cases}
\]
is not a continuous function. This example can easily be adapted to obtain a discontinuous infimal convolution for smooth functions \(f\) and \(g\). We conjecture that if the domain is assumed convex, and if both functions are Lipschitz continuous, and their domain is of non-empty interior, then their infimal convolution is Lipschitz continuous.

Lemma \ref{le:full_lipschitz_continuity}, together with Lemma \ref{le:hamilton_jacobi} and Rademacher's theorem, prove the following proposition:
\begin{prop}[Hamilton-Jacobi]\label{pr:hamilton_jacobi}
  Let \(g,W:\R^n\to\R\). If \(g\) is nonnegative, locally Lipschitz continuous, and \(W\) is convex and  grows superlinearly,
  \[
  \lim_{\abs x\to+\infty}\frac{W(x)}{\abs x}=+\infty,
  \]
  then, for almost every \(h\geq 0\) and \(x\in\R^n\),
  \[
  \frac{\partial}{\partial h}Q_h^W(g)(x)=-W^*(\nabla Q_h^Wg(x)).
  \]  
\end{prop}


\subsection{An equivalent formulation of the classical Borell-Brascamp-Lieb inequality}

In this subsection, we prove an interesting equivalence between the classical Borell-Brascamp-Lieb inequality and its differentiated expression, as announced in the introduction. It is also a good presentation of what is to come in the following sections.
\begin{prop}
  Let \(g,W:\R^n\to\R\). If \(g\) is nonnegative, locally Lipschitz continuous, and \(W\) is convex and grows superlinearly,
  \[
  \lim_{\abs x\to+\infty}\frac{W(x)}{\abs x}=+\infty,
  \]
  and are such that \(\int g^{-n}=\int W^{-n}=1\), and if $(g,W)$ is admissible in the sense of Definition \ref{de:admissibility}, then the following statements are equivalent:
  \begin{enumerate}
  \item The Borell-Brascamp-Lieb inequality holds: for every \(t\in[0,1]\) and \(H:\R^n\to\R\) such that
    \[
    \forall x,y\in\R^n,\quad H((1-t)x+ty)\leq (1-t)g(x)+tW(y),
    \]
    there holds
    \[
    \int H^{-n} \geq 1.
    \]
  \item The following inequality stands:
    \begin{equation*}
      \int \frac{W^*(\nabla g)}{g^{n+1}}\geq 0.
    \end{equation*}
  \end{enumerate}
\end{prop}

\begin{proof}
  By definition of the infimal convolution \(Q_h^W(g)\), it is actually sufficient to only consider the function \(H=(1-t)Q_h^W(g)(\,.\,/(1-t))\), where \(h=t/(1-t)\), in statement \(a.\) In fact, this leads to the statement \(a'.\):
  \[
  \int Q_h^W(g)^{-n}\geq 1,
  \]
  which we prove is equivalent to \(b.\) 
  
  Let us consider the function \(\phi:h\mapsto \int Q_h^W(g)^{-n}\), which is continuous and almost everywhere differentiable in light of Lemma \ref{le:full_lipschitz_continuity} and Theorem \ref{th:adm_cv} in the Appendix.
  Its derivative is given by
  \[
  \phi'(h) = n\int \frac{W^*(\nabla g)}{g^{n+1}}.
  \]
  The implication \(a'.\implies b.\) follows from the fact that \(\phi(0)=1\), and \(\phi(h)\geq1\) for \(h\geq 0\). Then, necessarily, \(\phi'(0)\geq0\).

  Conversely, assume that \(b.\) holds. Then, whenever \(h>0\) is such that \(\phi(h)=\int Q_h^W(g)^{-n}= 1\), statement \(b.\) applied to the function \(\tilde g=Q_h^W(g)\) and the corresponding function \(\tilde\phi\) implies that \(\tilde\phi'(0)=\phi'(h)\geq 0\) thanks to the semigroup property proved in Lemma \ref{le:semigroup}. This, together with the fact that \(\phi(0)=1\), proves that \(\phi\) stays above \(1\), which is exactly statement \(a.\)  
\end{proof}
Once again, we insist on the fact that the semigroup $Q_h^W$ is not linear, and not Markov, which means, in particular, that there is no mass conservation. As such, this result stands as a bit unusual among similar results.


\section{Sharp Gagliardo-Nirenberg-Sobolev inequalities}\label{se:sharp_GNS}

\subsection{Borell-Brascamp-Lieb}

Let us start from Theorem 8 in \cite{BCEFGG}, the dynamical formulation of Borell-Brascamp-Lieb inequality.

\begin{thrm}[\cite{BCEFGG}]\label{th:BBL}
  Let \(a>1\) and \(n\in\N^*\) such that \(a\geq n\), and \(g,W:\R^n\to(0,+\infty]\) be measurable functions such that \(\int g^{-a}=\int W^{-a}=1\). Then, for any \(h\geq 0\),
    \begin{equation}\label{eq:BBL}
      (1+h)^{a-n}\int_{\R^n}Q_h^W(g)^{1-a} \geq \int_{\R^n} g^{1-a} +h\int_{\R^n} W^{1-a},
    \end{equation}
    where
    \[
    Q_h^W(g)(x)=\inf_{y\in\R^n}\{g(x-hy)+hW(y)\}\in(0,+\infty].
    \]
    Furthermore, when \(g\) is equal to \(W\) and is convex, there is equality.
\end{thrm}
To see that there is equality whenever \(g=W\) is convex, fix \(x\in\R^n\). For any \(y\in\R^n\), since \(\frac{x}{1+h}=\frac{1}{1+h}(x-hy)+\frac{h}{1+h}y\),
\[
(1+h)\left(\frac{W(x-hy)}{1+h}+\frac{h}{1+h}W(y)\right)\geq (1+h)W\left(\frac{x}{1+h}\right).
\]
Conversely, $Q_h^W(g)(x)$ is achieved at \(y=x/(1+h)\). In particular, for all \(x\in\R^n\), \(h\geq0\),
\[
Q_h^W(W)(x)=(1+h)W\left(\frac{x}{1+h}\right),
\]
and equality in \eqref{eq:BBL} is a straightforward computation.

In \cite{BCEFGG}, subsection 3.2, Bolley, Cordero-Erausquin, Fujita, Gentil, and Guillin use Theorem \ref{th:BBL} to prove optimal Sobolev and Gagliardo-Nirenberg-Sobolev type inequalities in the half-space \(\R_n^+=\R^{n-1}\times \R_+\).
We want to extend these results to more general domains \(\Omega\) in \(\R^n\), where \(n\geq 2\). Let us assume that \(\Omega\) is the epigraph of a continuous function \(\varphi:\R^{n-1}\to\R\) such that \(\varphi(0)=0\). In other words,
\[
\Omega=\{(x_1,x_2)\in\R^{n-1}\times\R, x_2\geq\varphi(x_1)\}.
\]
Let \(e=(0,1)\in\R^{n-1}\times\R\), and for \(h\geq 0\), define
\[
\Omega_h=\Omega+\{he\}=\{(x_1,x_2)\in\R^{n-1}\times\R, x_2\geq\varphi(x_1)+h\}.
\]
Let \(a\geq n\), and consider \(g:\Omega\to(0,+\infty)\) and \(W:\Omega_1\to(0,+\infty)\), two measurable functions such that \(\int_\Omega g^{-a}=\int_{\Omega_1}W^{-a}=1\). After extending these functions by \(+\infty\) outside of their respective domain, inequality \eqref{eq:BBL} yields
\begin{equation}\label{eq:BBL2}
  (1+h)^{a-n}\int_{B_h}Q_h^W(g)^{1-a} \geq \int_{\Omega} g^{1-a} +h\int_{\Omega_1} W^{1-a}
\end{equation}
where
\[
B_h=\operatorname{dom}(Q_h^W(g)).
\]
When \(g(x)=W(x+e)\) and \(W\) is convex, then
\[
Q^W_h(g)(x)=(1+h)W\left(\frac{x+e}{1+h}\right)
\]
and equality is reached in the inequality above.

To get a sense of what is to follow, notice that there is equality in inequality \eqref{eq:BBL2} when \(h=0\). Now, when \(\Omega=\R^n_+\), the interesting fact that \(\Omega_h=B_h\) allows us, under certain admissibility criteria for \(W\) and \(g\), to compute the derivative of inequality \eqref{eq:BBL2} with respect to \(h\), at \(h=0\). By doing so, the term \(\int_{\partial \R_+^n} Q_0^W(g)^{1-a}=\int_{\partial \R_+^n} g^{1-a}\) appears in the left hand side, thus leading to trace inequalities.

Before going any further, let us investigate under which condition the two sets \(\Omega_h\) and \(B_h\) coincide. We have the following lemma:
\begin{lemm}\label{le:cone_support}
  There exists \(h_0>0\) such that for all \(h\in(0,h_0)\), \(B_h=\Omega_h\) if, and only if, \(\Omega\) is a convex cone. In that case, \(B_h\) and \(\Omega_h\) coincide for all \(h\geq 0\).
\end{lemm}

\begin{proof}
  First, note that \(Q_h^W(g)(x)<+\infty\) if, and only if, there exists \(y\in\Omega_1\) such that \(x-hy\in\Omega\). By definition of \(\Omega\), this is equivalent to
  \begin{align*}
    &\exists\, (y_1,y_2)\in\R^{n-1}\times\R \text{ s.t. }
    \left\{\begin{aligned}
    y_2&\geq \varphi(y_1)+1 \\
    x_2-hy_2&\geq \varphi(x_1-hy_1)
    \end{aligned} \right.\\
    \iff \Big(
    &\exists\, y_1\in\R^{n-1} \text{ s.t. }
    x_2\geq\varphi(x_1-hy_1)+h\varphi(y_1)+h)\Big).
  \end{align*}
  If \(x\in \Omega_h\), then choosing \(y_1=0\) proves that \(x\in B_h\), so \(\Omega_h\subset B_h\). If \(h>0\), \(\Omega_h=B_h\) if, and only if, for all \(x_1,y_1\in\R^{n-1}\),
  \begin{equation}\label{eq:cone_support}
    \varphi\left(\frac{x_1-y_1}{h}\right)\geq\frac{\varphi(x_1)-\varphi(y_1)}{h}. 
  \end{equation}
  Indeed, if \(\Omega_h\supset B_h\), then, for any \(x_1,y_1\in\R^{n-1}\), 
  \[
  x_2\coloneqq\varphi(x_1-hy_1)+h\varphi(y_1)+h\geq \varphi(x_1)+h
  \]
  and thus, replacing \(y_1\) by \((x_1-y_1)/h\), we get the stated inequality. The reciprocal is immediate.

  Now, let \(z\in\R^{n-1}\), \(\abs{z}=1\). Inequality \eqref{eq:cone_support}, for \(y_1=0\), becomes
  \[
  \varphi(z)\geq\frac{1}{h}\varphi(hz)
  \]
  for any \(h\) smaller than \(h_0\). Let \(\alpha=\limsup_{h\to 0}\varphi(hz)/h\). Using inequality \eqref{eq:cone_support} once again, we get, for any \(s\geq 0\),
  \[
  \varphi(sz)\geq\frac{s}{sh}\varphi(shz),
  \]
  for any sufficiently small \(h>0\). Taking the inferior limit when \(h\to 0\) proves that for any \(s\geq0\)
  \begin{equation}\label{eq:cone_support2}
  \varphi(sz)\geq s\alpha.
  \end{equation}
  The set \(\{s\geq 0,\varphi(sz)=s\alpha \}\) is non-empty because it contains \(0\), and it is closed by continuity. Let \(s\geq 0\) be such that \(\varphi(sz)=s\alpha\). Then, invoking inequality \eqref{eq:cone_support}, and then inequality \eqref{eq:cone_support2}, we get
  \begin{align*}
    \varphi\left(\frac{(1+h)sz-sz}{h}\right)=\varphi(sz)=s\alpha&\geq \frac{\varphi((1+h)sz)-\varphi(sz)}{h}\\
    &=\frac{\varphi((1+h)sz)-s\alpha}{h}\\
    &\geq \frac{(1+h)s\alpha-s\alpha}{h}= s\alpha
  \end{align*}
  so there is actually equality, and \(\varphi((1+h)sz)=(1+h)s\alpha\) for any sufficiently small \(h>0\). This shows that the connected component of \(\{s\geq 0,\varphi(sz)=s\alpha \}\) containing \(0\) is open in \(\R_+\). Since it is also closed, it is the half real line \(\R_+\). Thus, \(\varphi\) is linear over half-lines with initial point \(0\). Inequality \eqref{eq:cone_support} then becomes 
  \[
  \varphi(x_1-y_1)\geq \varphi(x_1)-\varphi(y_1)
  \]
  for any \(x_1,y_1\in\R^{n-1}\). Let \(t\in\left[0,1\right]\); replacing \(x_1\) by \((1-t)x_1+ty_1\) and \(y_1\) by \(ty_1\), and using linearity, the inequality becomes exactly the convexity inequality, that is
  \[
  \varphi((1-t)x_1+ty_1)\leq(1-t)\varphi(x_1)+t\varphi(y_1).
  \]
  The reciprocal is trivial. It is also clear that in this case, \(B_h=\Omega_h\) for \emph{any} \(h\geq0\).
\end{proof}

This lemma will be used in section \ref{se:convex_cones} to prove the trace Sobolev and the trace Gagliardo-Nirenberg-Sobolev inequalities in convex cones. We can go a bit further, and impose only \(\varphi\) to be convex. 

\begin{lemm}\label{le:convex_support}
  If \(\varphi\) is convex, then
  \[
  B_h=\left\{(x_1,x_2)\in\R^n, x_2\geq h+(1+h)\varphi\left(\frac{x_1}{1+h}\right)\right\}.
  \]
\end{lemm}

\begin{proof}
  One may notice that setting \(\omega(x)=0\) if \(x\in\Omega\) and \(+\infty\) if \(x\in\Omega^c\), and \(W(x)=\omega(x-e)\), then \(\omega\) is convex, thus
  \[
  B_h=\operatorname{dom}(Q_h^W(\omega))=\operatorname{dom}\left(x\mapsto(1+h)\,W\!\left(\frac{x+e}{1+h}\right)\right),
  \]
  and
  \begin{align*}
    W\left(\frac{x+e}{1+h}\right) < +\infty &\iff \frac{x+e}{1+h}-e \in\Omega\\
    &\iff x_2\geq h+(1+h)\varphi\left(\frac{x_1}{1+h}\right).
  \end{align*}
\end{proof}


\subsection{Convex cones}\label{se:convex_cones}
In this subsection, we assume that \(\Omega\) is a convex cone. In that case, invoking Lemma \ref{le:cone_support}, inequality~\eqref{eq:BBL} becomes
\begin{equation}\label{eq:convex_cone_BBL}
(1+h)^{a-n}\int_{\Omega_h}Q_h^W(g)^{1-a} \geq \int_{\Omega} g^{1-a} +h\int_{\Omega_1} W^{1-a},
\end{equation}
for any \(h>0\), and there is equality when \(h=0\). Taking the derivative of this inequality with respect to \(h\), under the admissibility conditions for \(g\) and \(W\) exposed in full details in Appendix~\ref{se:admissibility}, and evaluating at \(h=0\), we prove that
\begin{equation}\label{eq:derived_BBL}
  (a-n)\int_\Omega g^{1-a}+(a-1)\int_\Omega\frac{W^*(\nabla g)}{g^a}-\int_{\R^{n-1}}g^{1-a}(x_1,\varphi(x_1))\d x_1 \geq \int_{\Omega_1}W^{1-a}.
\end{equation}
There, we used Lemma \ref{le:hamilton_jacobi}, and the fact that
\begin{align*}
  \frac{1}{h}\left(\int_{\Omega_h} Q_h^W(g)^{1-a}-\int_\Omega g^{1-a}\right) &= \int_{\Omega_h} \frac{Q_h^W(g)^{1-a}-g^{1-a}}{h} + \frac{1}{h}\left(\int_{\Omega_h} g^{1-a}-\int_\Omega g^{1-a}\right)\\
  &\hspace{-3em}=\int_{\Omega_h} \frac{Q_h^W(g)^{1-a}-g^{1-a}}{h} - \frac{1}{h}\left(\int_{\R^{n-1}}\int_{\varphi(x_1)}^{h+\varphi(x_1)}g^{1-a}(x_1,x_2)\d x_2\d x_1\right) \\
  &\hspace{-3em}\xrightarrow[h\to 0]{} (1-a)\int_\Omega\frac{W^*(\nabla g)}{g^a}-\int_{\R^{n-1}}g^{1-a}(x_1,\varphi(x_1))\d x_1,
\end{align*}
see Theorem \ref{th:adm_cv}.

Let \(p\in(1,n)\), and \(q\) its conjugate exponent, \(1/p+1/q=1\). Applying inequality \eqref{eq:derived_BBL} to the function \(W\) defined by \(W(x)=C\norm x^q/q\), where \(C>0\) is such that \(\int W^{-a}=1\), which happens to be admissible for this choice of \(q\), in the sense of Definition \ref{de:admissibility} in the Appendix. We find
\[
(a-n)\int_\Omega g^{1-a}+C^{1-p}\frac{a-1}{p}\int_\Omega\frac{\norm{\nabla g}_*^p}{g^a}-\int_{\R^{n-1}}g^{1-a}(x_1,\varphi(x_1))\d x_1 \geq \int_{\Omega_1}W^{1-a}
\]
for any admissible \(g\), where \(\norm x_*=\sup_{\norm y=1}x\cdot y\) is the dual norm of \(x\). 
Next, we extend the above inequality to all functions \(g\) such that \(f=g^{(p-a)/p}\in\mc C^\infty_c(\Omega)\). This can be done by approximation by admissible functions, we refer to the Appendix \ref{se:admissibility}. Rewriting the quantities in terms of \(f=g^{-(a-p)/p}\) yields
\[
\int_{\R^{n-1}}f^{p\frac{a-1}{a-p}}(x,\varphi(x))\d x\leq{} C^{1-p}\frac{a-1}{p}\left(\frac{p}{a-p}\right)^p\int_\Omega\norm{\nabla f}^p_*-\int_{\Omega_1}W^{1-a}+(a-n)\int_\Omega f^{p\frac{a-1}{a-p}}
\]
We may then remove the normalization to find that inequality \eqref{eq:derived_BBL} becomes
\begin{equation}\label{eq:extended_BBL}
  \begin{aligned}
    \int_{\R^{n-1}}f^{p\frac{a-1}{a-p}}(x,\varphi(x))\d x\leq{}& C^{1-p}\frac{a-1}{p}\left(\frac{p}{a-p}\right)^p\left(\int_\Omega\norm{\nabla f}^p_*\right)\beta^{p\frac{p-1}{a-p}}-\left(\int_{\Omega_1}W^{1-a}\right)\beta^{p\frac{a-1}{a-p}}\\
    &+(a-n)\int_\Omega f^{p\frac{a-1}{a-p}}
  \end{aligned}
\end{equation}
where
\[
\beta=\left(\int_\Omega f^{\frac{pa}{a-p}}\right)^{\frac{a-p}{ap}}.
\]
Now, define \(u=\frac{a-1}{a-p}\) and \(v=u'=\frac{a-1}{p-1}\), so that \(u,v>1\) and \(1/u+1/v=1\). By Young's inequality, we find
\begin{equation}\label{eq:young}
  \begin{aligned}
    A\int_\Omega\norm{\nabla f}^p_*\beta^{p\frac{p-1}{a-p}}-\left(\int_{\Omega_1}W^{1-a}\right)\beta^{p\frac{a-1}{a-p}}
    &= Bv\left(\frac{A}{Bv}\int_\Omega\norm{\nabla f}^p_*\beta^{p\frac{p-1}{a-p}}-\frac{1}{v}\beta^{p\frac{a-1}{a-p}}\right)\\
    &\leq D\left(\int_\Omega\norm{\nabla f}^p_*\right)^{u},
  \end{aligned}
\end{equation}
where
\[
A=C^{1-p}\frac{a-1}{p}\left(\frac{p}{a-p}\right)^p,\quad B=\int_{\Omega_1}W^{1-a}\quad\text{and}\quad D=\frac{A^u}{(Bv)^{u-1}}\frac{1}{u}.
\]

In order to find a more compact inequality, we consider, for \(\lambda>0\), \(f_\lambda:x\mapsto f(\lambda x)\). By linearity of \(\varphi\), applying \eqref{eq:young} to \(f_\lambda\) leads to
\[
\int_{\R^{n-1}}f^{p\frac{a-1}{a-p}}(x,\varphi(x))\d x\leq \lambda^{(a-n)\frac{p-1}{a-p}}\frac{A^u}{(Bv)^{u-1}}\frac{1}{u}\left(\int_\Omega\norm{\nabla f}^p_*\right)^{u}+\frac{a-n}{\lambda}\int_\Omega f^{p\frac{a-1}{a-p}}.
\]
Optimizing this inequality with respect to \(\lambda>0\) finally yields inequality \eqref{eq:trace_GN_cone} of Theorem \ref{th:trace_GN_cone}

It remains to show that inequality \eqref{eq:trace_GN_cone} is optimal. The function for which equality is reached does not have compact support, but this technicality does not bear much relevance.
To prove optimality, note that there is equality in \eqref{eq:derived_BBL} when \(g(x)=W(x+e)\), which implies equality in \eqref{eq:extended_BBL} when \(f(x)=\norm{x+e}^{-\frac{a-p}{p-1}}\). If Young's inequality \eqref{eq:young} is an equality, then the optimization with respect to parameter \(\lambda\) necessarily preserves the equality. Thus, it is enough to show that for \(f(x)=\norm{x+e}^{-\frac{a-p}{p-1}}\), there is equality in \eqref{eq:young}. This is the case if, and only if,
  \[
  \frac{A}{Bv}\int_\Omega\norm{\nabla f}^p_*=\left(\beta^{p\frac{p-1}{a-p}}\right)^{v-1}.
  \]
  Let us now write, for \(\alpha>0\)
  \[
  I_\alpha\coloneqq\int_\Omega \norm{x+e}^{-\alpha}.
  \]
  Then,
  \[
  C=q\left(\int_\Omega \norm{x+e}^{-qa}\right)^{\frac{1}{a}}=\frac{p}{p-1}I_{ap/(p-1)}^{1/a}
  \]
  hence
  \[
  A=\frac{(a-1)(p-1)^{p-1}}{(a-p)^p}I_{ap/(p-1)}^{(1-p)/a},\quad
  B=I_{ap/(p-1)}^{(1-a)/a}I_{p(a-1)/(p-1)}, \quad\text{and}\quad
  \left(\beta^{p\frac{p-1}{a-p}}\right)^{v-1}=I_{ap/(p-1)}^{(a-p)/a}.
  \]
  \begin{claim}
    For \(\gamma\in\R\), let \(h:\R^n\backslash\{0\}\to \left]0,+\infty\right[, x\mapsto \norm{x}^{\gamma}\). Then, almost everywhere, \(h\) is differentiable, and \(\norm{\nabla h(x)}^p_*=\abs\gamma\norm{x}^{\gamma-1}\).
  \end{claim}
  \noindent Using this, we conclude that there is indeed equality in \eqref{eq:young}, since then
  \[
  \int_\Omega\norm{\nabla f}^p_*=\left(\frac{a-p}{p-1}\right)^pI_{p(a-1)/(p-1)}.
  \]
\begin{proof}[Proof of the claim.]
  Consider \(\phi:x\mapsto \norm x\) and \(\psi:\rho\mapsto\rho^\gamma\). \(\phi\) is convex, hence almost everywhere differentiable by Rademacher's theorem, and \(\psi\) smooth on \(\left]0,+\infty\right[\), hence the claimed regularity of \(h=\psi\circ\phi\). For almost every \(x\), \(\nabla h(x)=\gamma \nabla\phi(x)\norm x^{\gamma-1}\), so
   \[
   \norm{\nabla h(x)}_*=\abs\gamma\norm x^{\gamma-1}\norm{\nabla\phi(x)}_*
   \]
   If \(x\neq0\) is a point of differentiability of \(\phi\), and \(t>0\), then
   \[
   1=\frac{\norm{x+tx/\norm x}-\norm x}{t}\xrightarrow[t\to0]{}\nabla\phi(x)\cdot\frac{x}{\norm x},
   \]
   so \(\norm{\nabla \phi(x)}_*\geq 1\). Conversely, if \(\norm v=1\), then
   \[
   \nabla\phi(x)\cdot v=\lim_{t\to 0^+}\frac{\norm{x+tv}-\norm x}{t}\leq \lim_{t\to 0^+}\norm v =1,
   \]
   so \(\norm{\nabla \phi(x)}_*=1\) and the claim is proved.
\end{proof}


\subsection{Convex sets}
Let us now assume that \(\Omega\) is the epigraph of a convex function \(\varphi\), with \(\varphi(0)=0\). Then, according to Lemma \ref{le:convex_support}, for \(h\geq 0\),
\[
B_h=\operatorname{dom}(Q_h^W(g))=\left\{(x_1,x_2)\in\R^n, x_2\geq h+(1+h)\varphi\left(\frac{x_1}{1+h}\right)\right\}.
\]
Inequality \eqref{eq:BBL} becomes
\begin{equation}\label{eq:convex_BBL}
(1+h)^{a-n}\int_{B_h}Q_h^W(g)^{1-a} \geq \int_{\Omega} g^{1-a} +h\int_{\Omega_1} W^{1-a},
\end{equation}
and there still is equality for all \(h>0\) whenever \(g(x)=W(x+e)\) and is convex. However, it is slightly trickier to compute the derivative at \(h=0\), since \(B_h\neq \Omega_h\), and their symmetric difference depends heavily on \(\varphi\). Effectively, a third term appears when trying to differentiate \(\int_{B_h}Q^h_W(g)^{1-a}\):
\[
\frac{1}{h}\left(\int_{B_h}Q_h^W(g)^{1-a}-\int_{\Omega}g^{1-a}\right) =
\int_{\Omega_h}\frac{Q_h^W(g)^{1-a}-g^{1-a}}{h}
-\frac{1}{h}\int_{\Omega\backslash\Omega_h}g^{1-a}
+\frac{1}{h}\int_{B_h\backslash\Omega_h} Q_h^W(g)^{1-a}.
\]
Taking the derivative at \(h=0\), when possible, yields
\begin{equation}\label{eq:derived_convex_BBL}
    (a-n)\int_\Omega g^{1-a}+(a-1)\int_\Omega\frac{W^*(\nabla g)}{g^a}-\int_{\R^{n-1}}g^{1-a}(x_1,\varphi(x_1))P(x_1)\d x_1 \geq \int_{\Omega_1}W^{1-a},
\end{equation}
where
\[
P(x_1)=1+\varphi(x_1)-x_1\cdot\nabla\varphi(x_1).
\]
\begin{rmrk}
  To prove this, we had to assume that \(\varphi\) satisfies some growth condition which will be made explicit in the next theorem. The strict generality cannot be preserved here, as \(\int_{\R^{n-1}}g^{1-a}P\) may not be integrable for certain choices of \(\varphi\), where \(g\) is assumed to be the optimal function. To nuance this, it might be possible to prove this result for such a choice of \(\varphi\) whenever \(g^{1-a}\) has compact support, but then, it is not obvious whether the inequality is still optimal.
\end{rmrk}
Using inequality \eqref{eq:derived_convex_BBL} with \(W=C\norm.^q/q\), and extending it for all \(f=g^{-(a-p)/p}\in\mc C^\infty_c(\Omega)\) just like we did for convex cones, and finally invoking Young's inequality, we get the theorem
\begin{thrm}\label{th:trace_ineq}
  Let \(a\geq n>p>1\), and \(\Omega=\{(x_1,x_2)\in\R^{n-1}\times\R,x_2\geq\varphi(x_1)\}\) be a convex set. Assume that there exist some constants \(C>0\) and \(R>0\) such that
  \[
  \forall\, \norm{x_1}>R,\quad \abs{x_1\cdot\nabla\varphi(x_1)}\leq C\norm{(x_1,\varphi(x_1))}.
  \]
  Then, there exists a positive constant $D'_{n,p,a}(\Omega)$ such that for any positive function \(f\in C^\infty_c(\Omega)\),
  \begin{equation}\label{eq:trace_ineq}
    \int_{\R^{n-1}}f^{p\frac{a-1}{a-p}}(x,\varphi(x))P(x)\d x\leq D'_{n,p,a}(\Omega) \left(\int_\Omega\norm{\nabla f}_*^p\right)^{\frac{a-1}{a-p}}+(a-n)\int_\Omega f^{p\frac{a-1}{a-p}},
  \end{equation}
  where \(P(x)=1+\varphi(x)-x\cdot\nabla\varphi(x)\).
  Furthermore, when \(f(x)=\norm{x+e}^{-\frac{a-p}{p-1}}\), then \eqref{eq:trace_ineq} is an equality.
\end{thrm}
Applying this theorem for \(a=n\), we find a new version of the trace Sobolev inequality, Theorem \ref{th:trace_sob}, with $D'_{n,p}(\Omega)=D'_{n,p,n}(\Omega)$. It is important to note that in Theorem \ref{th:trace_sob}, aswell as in Theorem \ref{th:trace_ineq}, the left-hand side can be negative. The weight \(P\) itself generally \emph{is} negative outside of a compact neighbourhood of the origin, but the inequality is still optimal.


\appendix

\section{Admissibility}\label{se:admissibility}

In this section, we prove that the results are true for a class of admissible functions, and we extend these results to the appropriate, more general setting, by approximation by admissible functions. The difficulty here lies in that \(g\) must not be bounded or even Lipschitz, since \(g^{-a}\) has to be integrable. The case of the half-plane has already been investigated (in \cite{BCEFGG}), and easily extends to convex cones. Here, we will only tackle convex sets, which, although more technical, follows the same general idea.

Throughout this section, \(\varphi:\R^{n-1}\to [0,+\infty)\)  is a convex function such that \(\varphi(0)=0\), \(g:\Omega\to(0,+\infty)\) is assumed to be locally Lipschitz continuous, and \(W:\Omega_1\to(0,+\infty)\) is convex.

\subsection{Differentiating the Borell-Brascamp-Lieb inequality}
Inequality \eqref{eq:convex_BBL},
\[
(1+h)^{a-n}\int_{B_h}Q_h^W(g)^{1-a} \geq \int_{\Omega} g^{1-a} +h\int_{\Omega_1} W^{1-a},
\]
is trivially an equality for \(h=0\), we thus ask compute its derivative. Let us first give a non-rigorous proof for clarity. The most difficult part is computing the derivative of \(\int_{B_h}Q_h^W(g)^{1-a}\), so let us start with that. Notice that \(\Omega_h\subset B_h\cap\Omega\), thus
\[
\frac{1}{h}\left(\int_{B_h}Q_h^W(g)^{1-a}-\int_{\Omega}g^{1-a}\right) =
\underbrace{\int_{\Omega_h}\frac{Q_h^W(g)^{1-a}-g^{1-a}}{h}}_{(i)}
-\underbrace{\frac{1}{h}\int_{\Omega\backslash\Omega_h}g^{1-a}}_{(ii)}
+ \underbrace{\frac{1}{h}\int_{B_h\backslash\Omega_h} Q_h^W(g)^{1-a}}_{(iii)}.
\]
Recalling Lemma \ref{le:hamilton_jacobi}, almost everyhere,
\[
\lim_{h\to0}\frac{Q_h^W(g)(x)-g(x)}{h}=-W^*(\nabla g(x)),
\]
thus \((i)\) should converge towards
\[
(a-1)\int_{\Omega}\frac{W^*(\nabla g)}{g^a}.
\]
Next, \((ii)\) can be rewritten in a way such that the convergence is quite clear:
\[
(ii)=\int_{\R^{n-1}} \left(\frac{1}{h}\int_{\varphi(x_1)}^{\varphi(x_1)+h}g^{1-a}(x_1,x_2)\d x_2\right) \d x_1\xrightarrow[h\to0]{} \int_{\R^{n-1}} g^{1-a}(x_1,\varphi(x_1))\d x_1
\]
as \(h\to 0\). Finally, giving \((iii)\) the same treatment,
\begin{align*}
  (iii)=&\int_{\R^{n-1}}\left(\frac{1}{h}\int_{h+(1+h)\varphi(x_1/(1+h))}^{h+\varphi(x_1)}Q^W_h(g)^{1-a}(x_1,x_2)\d x_2\right) \d x_1 \\
  &\xrightarrow[h\to0]{} \int_{\R^{n-1}} g^{1-a}(x_1,\varphi(x_1))(x_1\cdot\nabla \varphi(x_1)-\varphi(x_1))\d x_1,
\end{align*}
since \(Q_0^W(g)=g\) and
\[
\lim_{h\to0}\frac{1}{h}\left(\varphi(x_1)-(1+h)\varphi\left(\frac{x_1}{1+h}\right)\right)=x_1\cdot\nabla \varphi(x_1)-\varphi(x_1).
\]
Summing these results up, we find the claimed derivative at \(h=0\). Whenever \(\Omega\) is a convex cone, \(B_h\backslash\Omega_h=\emptyset\), and thus \((iii)=0\). In that case, the argument is much more succinct, but since it is also a corollary of the more general case, we will not address it. The conditions for the convergence to play out nicely are summed up in the following definition. They are mostly growth conditions on \(g\) and \(W\), and will come into play later on.
\begin{defi}\label{de:admissibility}
  The couple of functions \((g,W)\) is said to be admissible if the following conditions are satisfied for some constant \(\gamma\):
  \begin{itemize}
    \setlength\itemsep{0em}
    \setlength\itemindent{1em}
  \item[(C0)\quad] \(\gamma>\max\left(\frac{a}{n-1},1\right)\);
  \item[(C1)\quad] there exists \(A_1>0\) such that \(W(x)\geq A_1\norm x^\gamma\) for all \(x\in\Omega_1\);
  \item[(C2)\quad] there exists \(A_2>0\) such that \(W(x)\leq A_2(1+\norm x^\gamma)\) for all \(x\in\Omega_1\);
  \item[(C3)\quad] there exists \(A_3>0\) such that \(g(x)\geq A_3(1+\norm x^\gamma)\) for all \(x\in\Omega\);
  \item[(C4)\quad] there exists \(A_4>0\) such that \(\norm{\nabla g(x)}\leq A_4(1+\norm x^{\gamma-1})\) for all \(x\in\Omega\).
  \end{itemize}
\end{defi}
The challenge is to prove that under these conditions, \(Q_h^W(g)\) converges towards \(g\) in a controlled manner as \(h\to 0\). 
The main result of this section is the following:
\begin{thrm}\label{th:adm_cv}
  Assume that the couple \((g,W)\) is admissible, and that there exist some constants \(C>0\) and \(R>0\) such that 
  \begin{equation}\label{eq:phi_growth}
    \forall\, \norm{x_1}>R,\quad \abs{x_1\cdot\nabla\varphi(x_1)}\leq C\norm{(x_1,\varphi(x_1))}.
  \end{equation}
  Then
  \begin{equation}
    \lim_{h\to0} \frac{1}{h}\left(\int_{B_h}Q_h^W(g)^{1-a}(g) -\int_\Omega g^{1-a}\right)= (a-1)\int_\Omega\frac{W^*(\nabla g)}{g^a}-\int_{\R^{n-1}}g^{1-a}(x_1,\varphi(x_1))P(x_1)\d x_1,
  \end{equation}
  where \(P(x_1)=1+\varphi(x_1)-x_1\cdot\nabla\varphi(x_1)\).
\end{thrm}

In what follows, we will use a good number of different positive constants, which will all be written \(C\) for convenience. They will not depend on \(x\in\R^n\), or \(h>0\), but might depend on \(A_i\), \(i\in\{1,2,3,4\}\), \(\gamma\).


\subsubsection{Convergence of \texorpdfstring{$(i)$}{(i)} and \texorpdfstring{$(ii)$}{(ii)}}
\begin{lemm}\label{le:easy_cv}
  If \((g,W)\) is admissible, there exist constants \(C>0\) and \(h_0>0\), such that for all \(0<h<h_0\), and \(x\in \Omega_h\),
  \[
  \abs{Q_h^W(g)(x)-g(x)}\leq Ch(1+\norm x^\gamma).
  \]
\end{lemm}
\begin{proof}
  First, let \(x',x\in\Omega\). Then, we may estimate \(\abs{g(x')-g(x)}\) using hypothesis (C4):
  \begin{align}
    \abs{g(x')-g(x)} &\leq \int_0^1 \norm{\frac{\partial}{\partial\theta}g(x+\theta(x'-x))}\d\theta \notag\\
    &\leq \norm{x'-x}\int_0^1 A_4(1+\norm{x+\theta(x'-x)}^{\gamma-1})\d\theta \notag\\
    &\leq C\norm{x'-x}\left(1+\norm x^{\gamma-1}+\norm{x'-x}^{\gamma-1}\right). \label{eq:g_growth}
  \end{align}
  Now, let \(0<h\leq1\) and \(x\in\Omega_h\). Then, \(x-he\in\Omega\), so  
  \begin{align*}
    Q_h^W(g)(x)-g(x) &\leq g(x-he)+hW(e)-g(x)\\
    &\leq Ch(1+\norm x^{\gamma -1}+h^{\gamma-1})+hW(e)\\
    &\leq Ch(1+\norm x^{\gamma-1}).
  \end{align*}
  For the converse inequality, we will of course use hypotheses (C1) and (C3), but we first have to localize the point where the infimum \(Q_h^W(g)(x)\) is reached. Let \(y\in\Omega_1\) be such that \(Q_h^W(g)(x)=g(x-hy)+hW(y)\). Then, invoking hypothesis (C1) and inequality \eqref{eq:g_growth},
  \begin{align*}
    hA_1\norm y^\gamma \leq hW(y) &= Q_h^W(g)(x)-g(x-hy) \\
    &=Q_h^W(g)(x)-g(x)+g(x)-g(x-hy) \\
    &\leq Ch(1+\norm x^{\gamma-1}) + Ch\norm y(1+\norm x^{\gamma -1}+(h\norm y)^{\gamma-1}).
  \end{align*}
  We thus choose \(h_0\in(0,1)\) such that for any \(h\in(0,h_0)\), \(A_1-Ch^{\gamma-1} > h^{\gamma-1}\). Then, for any \(h\in(0,h_0)\),
  \begin{align*}
    h^{\gamma-1}\norm y^\gamma & < (A_1-Ch^{\gamma-1})\norm y^\gamma\\
    &\leq C(1+\norm y)(1+\norm x^{\gamma-1}),
  \end{align*}
  which implies that
  \[
  h^{\gamma-1}\norm y^{\gamma-1} \leq C(1+\norm x^{\gamma-1}),
  \]
  since \(\norm y^{\gamma-1}\leq \max\left(1,2\frac{\norm y^\gamma}{1+\norm y}\right)\).
  Now, using inequality \eqref{eq:g_growth} once again,
  \begin{align*}
    \abs{g(x-hy)-g(x)}&\leq Ch\norm y(1+\norm x^{\gamma-1}+h^{\gamma-1}\norm y^{\gamma-1})\\
    &\leq Ch\norm y(1+\norm x^{\gamma-1}).
  \end{align*}
  Plugging this in the definition of \(Q_h^W(g)(x)\), we find
  \begin{align*}
    Q_h^W(g)(x)-g(x) &\geq \inf_{y\in\Omega_1} \left\{ -Ch\norm y(1+\norm x^{\gamma-1}) + hA_1\norm y^\gamma \right\}\\
    &\geq \inf_{y\in\R^n}\{\dots\} = -Ch(1+\norm x^\gamma).
  \end{align*}
  To conclude, it is enough to notice that \(1+\norm x^{\gamma-1}\leq 2+\norm x^\gamma\) since $\gamma >1$.
\end{proof}

Now that we have this estimation, we may estimate the speed of convergence of \(Q_h^W(g)^{1-a}\) towards \(g^{1-a}\).
\begin{prop}\label{pr:easy_cv}
  If \((g,W)\) is admissible, there exist constants \(C>0\) and \(h_0>0\), such that for all \(0<h<h_0\), and \(x\in \Omega_h\),
  \[
  \frac{\abs{Q_h^W(g)^{1-a}(x)-g^{1-a}(x)}}{h}\leq \frac{C}{1+\norm x^{\gamma(a-1)}}.
  \]  
\end{prop}
\begin{proof}
  First, let \(\alpha,\beta>0\). Then,
  \[
  \abs{\int_\alpha^\beta t^{-a}\d t}=\abs{\frac{1}{1-a}(\beta^{1-a}-\alpha^{1-a})}\leq \max(\alpha^{-a},\beta^{-a})\abs{\alpha-\beta},
  \]
  implying that
  \begin{equation}\label{eq:pow_ineq}
    \abs{\alpha^{1-a}-\beta^{1-a}}\leq (a-1)\abs{\alpha-\beta}(\alpha^{-a}+\beta^{-a}).
  \end{equation}
  Then, according to Lemma \ref{le:easy_cv}, there exists \(h_0>0\) such that for any \(h\in(0,h_0)\), and any \(x\in\Omega_h\),
  \begin{align}
    \frac{\abs{Q_h^W(g)^{1-a}(x)-g^{1-a}(x)}}{h}&\leq C\frac{\abs{Q_h^W(g)(x)-g(x)}}{h}\left(Q_h^W(g)^{-a}(x)+g^{-a}(x)\right)\notag\\
    &\leq C(1+\norm x^\gamma)\left(Q_h^W(g)^{-a}(x)+g^{-a}(x)\right). \label{eq:easy_cv1}
  \end{align}
  Now, hypotheses (C1) and (C3) and a straightforward computation yield
  \begin{align*}
    Q_h^W(g)(x)&\geq \inf_{y\in\Omega_1}\{A_3(1+\norm{x-hy}^\gamma)+hA_1\norm y^\gamma\}\\
    &\geq \inf_{y\in\R^n}\{A_3(1+\abs{\norm x-h\norm y}^\gamma)+hA_1\norm y^\gamma\}\\
    &\geq C(1+\norm x^\gamma).
  \end{align*}
  Using (C3) once again, we know that
  \[
  g^{-a}(x)\leq \left(A_3(1+\norm x^\gamma)\right)^{-a};
  \]
  putting these two inequalities together with inequality \eqref{eq:easy_cv1}, we finally obtain
  \begin{align*}
    \frac{\abs{Q_h^W(g)^{1-a}(x)-g^{1-a}(x)}}{h}&\leq C\frac{1+\norm x^\gamma}{(1+\norm x^\gamma)^a}\\
    &\leq  \frac{C}{1+\norm x^{\gamma(a-1)}}.
  \end{align*}  
\end{proof}
Proposition \ref{pr:easy_cv}, together with Lemma \ref{le:hamilton_jacobi}, proves the dominated convergence, and
\[
\lim_{h\to0}(i)=(a-1)\int_\Omega\frac{W^*(\nabla g)}{g^a},
\]
as claimed. The convergence of \((ii)\) is straightforward, as it is a direct implication of the local Lipschitz continuity of \(g\) and hypothesis (C3).


\subsubsection{Convergence of \texorpdfstring{$(iii)$}{(iii)}}
This term is a bit trickier, because comparing \(Q^W_h(g)\) to \(g\) is not possible on the entirety of \(B_h\), \(g\) being defined only on \(\Omega\). For many functions \(\varphi\), \(B_h\not\subset\Omega\) as is showcased on figure \ref{fi:domains} below.
\begin{figure}[h!]
  \centering
  \includegraphics{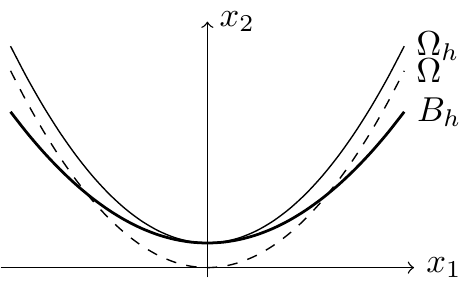}
  \caption{Graph of \(\Omega\), \(\Omega_h\), and \(B_h\) for \(\varphi(x_1)=\Vert x_1\Vert^2\) and \(h=0.5\)}
  \label{fi:domains}
\end{figure}
Thus, we prove the following result:
\begin{lemm}\label{le:hard_cv}
  If \((g,W)\) is admissible, there exist constants \(C>0\) and \(h_1>0\), such that for all \(0<h<h_1\), and \((x_1,x_2)\in B_h\backslash\Omega_h\),
  \begin{equation}\label{eq:hard_cv}
    \abs{Q_h^W(g)(x_1,x_2)-g(x_1,\varphi(x_1))}\leq hC(1+\norm{(x_1,\varphi(x_1))}^\gamma + \abs {x_1\cdot\nabla\varphi(x_1)}^\gamma).
  \end{equation}
\end{lemm}
The proof follows the same logic as the proof of Lemma \ref{le:easy_cv}.
\begin{proof}
  Recall that, according to Lemma \ref{le:convex_support}
  \[
  \Omega_h=\left\{(x_1,x_2)\in\R^n,\,x_2\geq h+\varphi\left(x_1\right)\right\},\quad B_h=\left\{(x_1,x_2)\in\R^n,\,x_2\geq h+(1+h)\varphi\left(\frac{x_1}{1+h}\right)\right\},
  \]
  and that 
  \[
  \abs{g(x')-g(x)} \leq C\norm{x'-x}\left(1+\norm x^{\gamma-1}+\norm{x'-x}^{\gamma-1}\right) \tag{\ref{eq:g_growth} revisited}
  \]
  for any \(x',x\in\Omega\).

  \textbf{1.} Fix \(h\in(0,1)\), \(x=(x_1,x_2)\in B_h\backslash\Omega_h\), and define \(p(x_1,x_2)=(x_1,\varphi(x_1))\), its projection onto \(\partial\Omega\). Letting \(y=\left(\frac{x_1}{1+h},1+\varphi\left(\frac{x_1}{1+h}\right)\right)\), we find that \(y\in\Omega_1\), and also that \(x-hy\in\Omega\),  thus, with hypothesis (C2) and inequality \eqref{eq:g_growth},
  \begin{align*}
    Q_h^W(x)-g(p(x))&\leq g(x-hy)-g(p(x))+hW(y)\\
    &\leq C\norm{x-hy-p(x)}\left(1+\norm{p(x)}^{\gamma-1}+\norm{x-hy-p(x)}^{\gamma-1}\right)+hA_2(1+\norm y^\gamma).
  \end{align*}
  For brevity, let us write \(u=\norm{p(x)}=\norm{(x_1,\varphi(x_1)}\), and \(v=x_1\cdot\varphi(x_1)=\abs{x_1\cdot\varphi(x_1)}\).
  Now, notice that
  \[
  \norm{x-hy-p(x)}=h\norm{\left(\frac{x_1}{1+h},\frac{x_2-h-\varphi(x_1)}{h}-\varphi\left(\frac{x_1}{1+h}\right)\right)}.
  \]
  From the definition of \(\Omega_h\) and \(B_h\), we find out that
  \begin{align*}
    0\leq \frac{h+\varphi(x_1)-x_2}{h} &\leq \frac{\varphi(x_1)-(1+h)\varphi(x_1/(1+h))}{h}\\
    &\leq x_1\cdot\nabla\varphi(x_1)=v,
  \end{align*}
  since \(\varphi\) is convex and nonnegative. Thus,
  \begin{align*}
    \norm{x-hy-p(x)}&\leq h\norm{\left(\frac{x_1}{1+h},\varphi\left(\frac{x_1}{1+h}\right)\right)}+h\abs{x_1\cdot\nabla\varphi(x_1)}\\
    &\leq h(\norm{p(x)}+\abs{x_1\cdot\nabla\varphi(x_1)})=h(u+v),
  \end{align*}
  so, since \(h<1\),
  \begin{align*}
    1+\norm{p(x)}^{\gamma-1}+\norm{x-hy-p(x)}^{\gamma-1}&\leq 1+u^{\gamma-1}+(h(u+v))^{\gamma-1}\\
    &\leq C(1+u^{\gamma-1}+v^{\gamma-1}).
  \end{align*}
  Finally,
  \[
  A_2(1+\norm y^\gamma) = A_2\left(1+\norm{\frac{x_1}{1+h},1+\varphi\left(\frac{x_1}{1+h}\right)}^\gamma\right)\leq C(1+u^\gamma).
  \]
  Putting all these inequalities together, we find   
  \begin{align}
    Q_h^W(x)-g(p(x)) &\leq hC(u+v)(1+u^{\gamma-1}+v^{\gamma-1})+hC(1+u^\gamma)\notag\\
    &\leq hC(1+u^\gamma+v^\gamma).\label{eq:hard_cv_upper}
  \end{align}
  
  \textbf{2.} Conversely, let \(y\in\Omega_1\) be such that \(Q_h^W(g)(x)=g(x-hy)+hW(y)\). As before, we localize \(y\). Using hypothesis \(A_1\) and inequalities \eqref{eq:g_growth} and \eqref{eq:hard_cv_upper},
  \begin{align*}
    hA_1\norm y^\gamma &\leq hW(y) = Q^W_h(g)(x)-g(p(x))+g(p(x))-g(x-hy)\\
    &\leq hC(1+u^\gamma+v^\gamma) +C\norm{x-hy-p(x)}\left(1+u^{\gamma-1}+\norm{x-hy-p(x)}^{\gamma-1}\right)\\
    &\leq hC(1+u^\gamma+v^\gamma) + hC(\norm y+v)\left(1+u^{\gamma-1}+h^{\gamma-1}(\norm y+v)^{\gamma-1}\right)\\
    &\leq hC(1+u^\gamma+v^\gamma)+hC(\norm y+v)\left(1+u^{\gamma-1}+h^{\gamma-1}\norm y^{\gamma-1}+ v^{\gamma-1}\right).
  \end{align*}
  Rearranging the terms and dividing by \(h\) yields
  \begin{align*}
    A_1\norm y^\gamma -Ch^{\gamma-1}\norm y^{\gamma-1}(\norm y+v)&\leq C(1+u^\gamma+v^\gamma) + C(\norm y+v)\left(1+u^{\gamma-1}+v^{\gamma-1}\right)\\
    &\leq C(1+u+v+\norm y)\left(1+u^{\gamma-1}+v^{\gamma-1}\right).
  \end{align*}
  We must now split the reasoning in two cases: either \(\norm y \leq v\), in which case the conclusion follows, or \(\norm y \geq v\), and then \( A_1\norm y^\gamma -Ch^{\gamma-1}\norm y^{\gamma-1}(\norm y+v)\geq  A_1\norm y^\gamma -2Ch^{\gamma-1}\norm y^\gamma\). We thus choose \(0<h_1<1\) such that for all \(h\in (0,h_1)\), \(A_2-2Ch^{\gamma-1}\geq h^{\gamma-1}\). Then, we have, for any \(h\in(0,h_1)\),
  \[
  \frac{h^{\gamma-1}\norm y^\gamma}{1+u+v+\norm y}\leq C\left(1+u^{\gamma-1}+v^{\gamma-1}\right).
  \]
  Once again, either \(\norm y\leq 1+u+v\), or
  \[
  h^{\gamma-1}\norm y^{\gamma-1}\leq\frac{2h^{\gamma-1}\norm y^\gamma}{1+u+v+\norm y}.
  \]
  Taking the greatest of the constants in those two cases, we may conclude that
  \begin{equation}\label{eq:hard_cv_loc}
    h^{\gamma-1}\norm y^{\gamma-1}\leq C(1+u^{\gamma-1}+v^{\gamma-1}).
  \end{equation}
  
  \textbf{3.} 
  We may now proceed with the converse inequality. Invoking once again inequality \eqref{eq:g_growth}, and then inequality \eqref{eq:hard_cv_loc},
  \begin{align*}
    \abs{g(x-hy)-g(p(x))}&\leq hC(\norm y+v)(1+u^{\gamma-1}+h^{\gamma-1}\norm y^{\gamma-1}+v^{\gamma-1})\\
    &\leq hC(\norm y+v)(1+u^{\gamma-1}+v^{\gamma-1}).
  \end{align*}
  Finally, 
  \begin{align*}
    Q_h^W(g)(x)-g(p(x))&=g(x-hy)-g(p(x))+hW(y)\\
    &\geq -hC(\norm y+v)(1+u^{\gamma-1}+v^{\gamma-1})+hA_2\norm y^{\gamma}\\
    &\geq h\inf_{y\in\R^n}\{-C(\norm y+v)(1+u^{\gamma-1}+v^{\gamma-1})+A_2\norm y^{\gamma}\}\\
    &\geq -hC\left(1+u^{\gamma-1}+v^{\gamma-1}\right)^{\gamma/(\gamma-1)},
  \end{align*}
  and we may conclude.
\end{proof}

We may now prove Theorem \ref{th:adm_cv}: using the same notations as in the proof above, that is \(u=\norm{p(x)}=\norm{(x_1,\varphi(x_1)}\), and \(v=x_1\cdot\varphi(x_1)=\abs{x_1\cdot\varphi(x_1)}\), hypothesis (C2) immediately yields, for all \(h>0\) and all \(x\in B_h\backslash\Omega_h\),
\[
g^{-a}(p(x)) \leq \frac{C}{(1+u^\gamma)^a}.
\]
Furthermore, inequality \eqref{eq:hard_cv} and hypothesis (C2) yield
\[
Q_h^W(g)(x)\geq -hC(1+u^\gamma+v^\gamma) + C(1+u^\gamma)
\]
for all \(x\in B_h\backslash\Omega_h\) and \(0<h<h_1\). Now, assumption \eqref{eq:phi_growth} reads: for all \(x_1\in\R^{n-1}\) such that \(\norm{x_1}>R\), 
\[
v\leq C u.
\]
Since both \(u\) and \(v\) are bounded functions of \(x\) on the set \(\{(x_1,x_2)\in B_h\backslash\Omega_h,\,\norm{x_1}\leq R\}\), there exists \(h_2>0\) such that, for all \(0<h<h_2\), 
\[
Q_h^W(g)(x)\geq \begin{cases} C > 0 & \text{whenever }\norm{x_1}\leq R \\ C(1+u^\gamma) & \text{whenever }\norm{x_1}> R \end{cases}
\]
Thus, for all \(0<h<h_2\) and all \(x\in B_h\backslash\Omega_h\),
\[
Q_h^W(g)^{-a}(x) \leq \frac{C}{(1+u^\gamma)^a}.
\]

Finally, invoking inequality \eqref{eq:pow_ineq} together with assumption \eqref{eq:phi_growth} yields, for any \(0<h<h_2\) and \(x=(x_1,x_2)\in B_h\backslash\Omega_h\),
\begin{align*}
  \frac{\abs{Q_h^W(g)^{1-a}(x)-g^{1-a}(p(x))}}{h} &\leq C \frac{\abs{Q_h^W(g)(x)-g(p(x))}}{h}\left(Q_h^W(g)^{-a}(x)+g^{-a}(p(x))\right)\\
  &\leq C(1+u^\gamma+v^\gamma)\frac{1}{(1+u^\gamma)^a}\\
  &\leq C\frac{1}{1+u^{(a-1)\gamma}}.
\end{align*}  

Note that \(u\geq \norm{x_1}\), and we chose \(a\) such that \((a-1)\gamma>n\), hence \(q(a-1)\gamma-q>q(n-1)\), thus the dominated convergence theorem applies, and we may conclude that
\begin{align*}
  \lim_{h\to0}\int_{B_h\backslash\Omega_h}\frac{1}{h}Q_h^W(g)^{1-a}&=\lim_{h\to0}\int_{\R^{n-1}}\left(\frac{1}{h}\int_{h+(1+h)\varphi(x_1/(1+h))}^{h+\varphi(x_1)}g^{1-a}(x_1,\varphi(x_1))\d x_2\right)\d x_1 \\
  &=\int_{\R^{n-1}}(x_1\cdot\nabla\varphi(x_1)-\varphi(x_1))g^{1-a}(x_1,\varphi(x_1))\d x_1,
\end{align*}
this last equality also being a dominated convergence result, using the hypotheses on \(g\).


\subsection{Extending the differentiated inequality}

We just proved that whenever \((g,W)\) is admissible, with \(\int_\Omega g^{-a}=\int_{\Omega_1}W^{-a}=1\), and \(\varphi\) satisfies the asymptotic growth condition \eqref{eq:phi_growth}, then
\begin{equation}\label{eq:diff_ineq}
  (a-n)\int_\Omega g^{1-a}+(a-1)\int_\Omega \frac{W^*(\nabla g)}{g^a}-\int_{\R^{n-1}}g^{1-a}(x_1,\varphi(x_1))P(x_1)\d x_1 \geq \int_{\Omega_1}W^{1-a}.
\end{equation}
Let \(q>1\). We want to use this inequality with \(W(x)=C\norm{x}^q/q\), where \(C>0\) is such that \(\int_{\Omega_1}W^{-a}=1\). The goal being to prove Sobolev-type inequalities, we may consider only the real \(q\) such that their conjugate exponent \(p=q/(q-1)\), which will appear in \(W^*\), is strictly less than \(n\). Thus, we assume that \(q>n/(n-1)\), and conditions (C0), (C1) and (C2) are automatically satisfied with \(\gamma=q\).

We now compute \(W^*\):
\begin{align}
  W^*(y)=\sup_{x\in\Omega_1}\{x\cdot y-C\norm{x}^q/q\}
  &\leq \sup_{x\in\R^n}\{x\cdot y-C\norm{x}^q/q\} \label{eq:legendre}\\
  \notag&=\sup_{R\geq 0}\sup_{\norm x=R}\{x\cdot y-C\norm{x}^q/q\}\\
  \notag&=\sup_{R\geq 0} \{R\norm y_*-CR^q/q\}\\
  \notag&=C^{1-p}\norm y_*^p/p.
\end{align}
It is important to note that \eqref{eq:legendre} becomes an equality for \(y=\nabla g(z)\) whenever \(g(\,.\,)=W(\,.+e)\), since in that case,
\[
 W^*(\nabla g(z))=\sup_{x\in\Omega_1}\{x\cdot \nabla g(z)-W(x)\}=\sup_{x\in\Omega}\{(x+e)\cdot \nabla g(z)-g(x)\}= e\cdot\nabla g(z)+g^*(\nabla g(z))
\]
and the supremum is indeed reached inside the right set. Optimality is not lost, and inequality \eqref{eq:diff_ineq} then becomes
\begin{equation}\label{eq:diff_ineq_bis}
  (a-n)\int_\Omega g^{1-a}+C^{1-p}\left(\frac{a-1}{p}\right)\int_\Omega \frac{\norm{\nabla g}^p}{g^a}-\int_{\R^{n-1}}g^{1-a}(x_1,\varphi(x_1))P(x_1)\d x_1 \geq \int_{\Omega_1}W^{1-a}.
\end{equation}

The next step is to lift the restrictions on the function \(g\), extending the results to more general functions. Our tool here will be approximation by admissible functions.




\subsubsection{\texorpdfstring{$f=g^{(p-a)/p}$}{f} is a smooth function with compact support} 

Let \(f\in C^\infty_c(\Omega)\) be a nonnegative function such that \(\int_\Omega f^{ap/(a-p)}=1\). Let us fix some \({\gamma>\max\{1,a/(n-1)\}}\) and consider, for \(\eps>0\),
\[
f_\eps(x)=\left(\eps\norm{x+e}^{-\gamma (a-p)/p}+C_\eps f\right),
\]
where \(C_\eps\) is such that \(\int_\Omega f_\eps^{ap/(a-p)} =1\), whenever \(\eps\) is small enough for \(C_\eps\) to exist.
It is not difficult to see that the corresponding functions \(g_\eps=f_\eps^{p/(p-a)}\) satisfy conditions (C3) and (C4), and that \(\int_\Omega g_\eps^{-a}=1\). Furthermore, \(C_\eps\) increases strictly as \(\eps\) decreases towards \(0\), and an argument of continuity shows that \(\lim_{\eps\to0}C_\eps=1\), meaning that, pointwise, \(\lim_{\eps\to0}g_\eps=f^{(p-a)/p}\eqqcolon g\). Finally, the dominated convergence theorem, applied to \(g_\eps^{1-a}=f_\eps^{(a-1)p/(a-p)}\) proves that inequality \eqref{eq:diff_ineq_bis} is indeed valid for \(g\). Rewriting it with \(f\) yields
\begin{equation}\label{eq:diff_ineq_ter}
  \begin{split}
    (a-n)\int_\Omega f^{p\frac{a-1}{a-p}}+C^{1-p}\left(\frac{a-1}{p}\right)\left(\frac{p}{a-p}\right)^p\int_\Omega \norm{\nabla f}^p-\int_{\R^{n-1}}f^{p\frac{a-1}{a-p}}(x_1,\varphi(x_1))P(x_1)\d x_1 \\
    \geq \int_{\Omega_1}W^{1-a}.
  \end{split}
\end{equation}


\subsection*{Acknowledgements}
This work was partly written while the author was visiting the beautiful Institute Mittag-Leffler in Stockholm, which we thank for its hospitality.
This research was supported by the French ANR-12-BS01-0019 STAB project.


\bibliographystyle{abbrv}
\bibliography{GNS}

\end{document}